\documentclass[12pt]{article}

\usepackage[T1]{fontenc}
\usepackage[american]{babel}
\usepackage[babel]{microtype}
\usepackage{cfr-lm}   

\usepackage[pdftex,%
    nomarginpar,%
    lmargin=1in,%
    rmargin=1in,%
    tmargin=1in,%
    bmargin=1in,%
]{geometry}                      

\usepackage[pdftex, final]{graphicx}  
\usepackage[pdftex, dvipsnames]{xcolor}  
\usepackage[margin=1cm]{caption}  
\usepackage{subcaption}           

\usepackage[leqno]{amsmath}   
\usepackage{amsthm}        
\usepackage{amssymb}

\usepackage{hyperref}     
\usepackage[capitalise, noabbrev]{cleveref}  
\hypersetup{
    colorlinks=true,
    linkcolor=NavyBlue,
    citecolor=NavyBlue
}              

\usepackage{thmtools, thm-restate}  
\newtheorem{theorem}{Theorem}[section]     
\newtheorem{lemma}[theorem]{Lemma}       

\newtheorem{proposition}[theorem]{Proposition}
\theoremstyle{definition}
\newtheorem{definition}[theorem]{Definition}

\newtheorem*{theorem*}{Theorem}      

\usepackage{chngcntr}  
\counterwithin*{equation}{subsection}
\numberwithin{equation}{section}              

\usepackage{enumitem}          
\setlist{noitemsep}           

\usepackage[backend=bibtex, url=false, doi=false, isbn=false, maxnames=10, style=numeric]{biblatex}
\title{Characterizing Hyperbolicity in Graphs}

\author{Faisal Leo Quraishi\footnotemark[1]}
\date{}

\begin{document}

\maketitle

\footnotetext[1]{Department of Mathematics and Statistics, University of Nevada, Reno, NV 89557}

\begin{abstract}

Gromov's delta-hyperbolicity, the classical measure of how tree-like a metric space is, works well on spaces with unbounded diameter but behaves poorly on finite graphs, where it depends primarily on diameter rather than geometry. We introduce a function relating Gromov's delta of a quadruple to its diameter and use it to define a normalized invariant that characterizes the hyperbolicity of any finite graph, taking values between zero (trees) and one (large lattice graphs). We derive a closed-form formula for this function on the hyperbolic plane. Using this formula, we give an alternate proof of the best constant of Gromov's delta-hyperbolicity for the hyperbolic plane. We also give the first theoretical proof of the optimal constant for the scaled Gromov four-point condition under diameter scaling, previously known only from numerical computations.

\end{abstract}

\section*{Introduction}

The study of hyperbolic geometry originated from analyzing the properties of the hyperbolic plane. However, many spaces share properties with the hyperbolic plane \parencite{Gromov1987}. These properties are known as hyperbolicity, and roughly correspond to how "tree-like" a space is  \cite{hamann2011treelikenesshyperbolicspaces}. Hyperbolicity has been found in many real-world systems, from the structure of the internet to clusters of earthquakes to biological data \cite{4445694} \cite{Boguna2010} \cite{Karla2020} \cite{Zhou2018}. Moreover, research into complex networks has shown that many of their defining properties emerge from an underlying hyperbolicity \cite{Krioukov2010}. Understanding this hyperbolicity yields many effective results. For example, complex networks generally have more efficient embedding into hyperbolic space than into Euclidean space \cite{Clough_2016} \cite{Boguna2010} \cite{chamberlain2017neuralembeddingsgraphshyperbolic}. Moreover, transport algorithms are optimized when they use hyperbolic coordinates instead of Euclidean ones \cite{Blasius2020}. 

Complex networks are examples of graphs. However, not all graphs exhibit an underlying hyperbolicity \cite{narayan2012lackhyperbolicityasymptoticerdosrenyi} \cite{gugelmann2012randomhyperbolicgraphsdegree}. Therefore, detecting hyperbolicity in graphs has become a topic of interest \cite{ex1}\cite{ex2}\cite{ex3}\cite{ex4}\cite{ex5}. Classically, this is done via Gromov's $\delta$-hyperbolicity on a metric space, where the distance between any two nodes is defined to be the length of the shortest path between them \cite{chen2013hyperbolicitysmallworldtreelikerandom} \cite{Borassi2015} \cite{fournier2015computinggromovhyperbolicitydiscrete}. In Gromov's $\delta$-hyperbolicity, the maximal value of $\delta(\square)$ (see \cref{def:Def1}) over all quadruples $\square$ in the metric space characterizes the degree of hyperbolicity. However, this method of using $\delta$
to detect geometry was developed for spaces with unbounded distances and does not account for the diameter of a finite metric space. In response to this, scaled approaches scale $\delta(\square)$ by a factor related to the size of $\square$ \cite{ex1} \cite{Jonckheere2011}. These methods give a bound to distinguish when a graph possesses hyperbolicity. However, these methods do not determine the degree of this hyperbolicity. 

Motivated by the need to characterize the degree of hyperbolicity in a graph, we develop in \cref{sec:def} a new method of evaluating the hyperbolicity of any general metric space. 
We do so by further analyzing the relationship between the value of $\delta(\square)$ and the diameter $D(\square)$ (see \cref{def:Def2}) of the quadruple $\square$. This relationship was studied empirically on the hyperbolic plane in \cite{Karla2020}. The limit of this relationship as the diameter approaches infinity was studied in \cite{Nica2016}, where they give a theoretical proof of the maximal $\delta(\square)$ (see \cref{sec:3.2} for details). The limit of this relationship as the diameter approaches $0$ was studied in \cite{Jonckheere2011}, where they give a theoretical framework and prove a numerical result for the maximal value of $\delta(\square)/D(\square)$ (see \cref{sec:3.3} for details). For the purposes of this paper, we name this relationship explicitly as:
\begin{restatable*}{definition}{GDef}
Let $X$ be a metric space with metric $d : X \times X \to \mathbb{R}$. We define $\text{Dist}_X$ to be the set of possible distances in $X$:
\[ \text{Dist}_X = \{d(a,b) : a,b \in X\} \subset [0,\infty).\]
 We define $\Gamma_X : \text{Dist}_X \to \mathbb{R}$ by letting $\Gamma_X(x)$ be the  supremum of $\delta(\square)$ over all quadruples $\square$ with diameter $x$: 
\[\Gamma_X(x) = \sup\limits_{\square \subset X : D(\square) = x} \delta(\square).\]
\end{restatable*} 
While the entire relationship $\Gamma_X$ is relevant to the geometry of a graph, it is useful to extract a numerical constant to quantify the degree of hyperbolicity in the graph. Thus, we characterize the hyperbolicity of a graph as follows:
\begin{restatable*}{definition}{MDef}
Given a finite connected weighted graph $G$ with the path metric, let 
\[s(G) = \frac{4\sum_{i \in \text{Dist}_G} \Gamma_G(i)}{D(G)(D(G)+1)}\]
where $D(G)$ is the diameter of $G$. 
\end{restatable*}
Our model has the following advantages: we have defined $s(G)$ such that $s(G)$ will always lie on the interval $[0,1]$. If $G$ is a tree, which models maximal hyperbolicity, then $s(G)$ will be zero. If $G$ is a lattice graph, our model of minimal hyperbolicity, then $s(G)$ approaches $1$ as the size of the lattice grows; specifically $s(L_n) = \frac{2n}{2n+1} \to 1$ (see \cref{sec:def}). This is notable because the classical maximal $\delta$ depends entirely on the diameter of the lattice graph, with $\delta(L_n) = n$ growing without bound as $n \to \infty$ \cite{chen2013hyperbolicitysmallworldtreelikerandom}, while our characterization's dependence on $n$ vanishes at rate $O(1/n)$. Moreover, we use our Main Theorem to motivate that $G$ has greater hyperbolicity than another graph $G'$ if $s(G) < s(G')$. Thus, we propose $s(G)$ as a characterization of hyperbolicity in a graph. 

Our main theorem finds the formula for $\Gamma_X$ when $X$ is the hyperbolic plane $\mathbb{H}^2$. 
\begin{restatable*}{theorem}{Main}\label{ther:Main}
Given the hyperbolic plane $\mathbb{H}^2$, we have
\[\Gamma_{\mathbb{H}^2}(x)  = x - \cosh^{-1}\left(\cosh^2\left(\frac{x}{2}\right)\right).\]
We also have $\Gamma_{\mathbb{H}^2}'(x) \geq 0$ and $\Gamma_{\mathbb{H}^2}''(x) \leq 0$ for $x \in (0,\infty)$. 
\end{restatable*}
In \cite{Karla2020}, the authors considered the maximum value of $\delta(\square)$ as a function of diameter $D(\square)$, in other words, $\Gamma_{\mathbb{H}^2}(x)$. They checked the value of this function for large samples of quadruples on the hyperbolic plane with various curvatures. In \cref{fig:KFig}, we compare our expression for $\Gamma_{\mathbb{H}^2}(x)$ in the Main Theorem to their statistical results. 
\begin{figure}[h]
\centering
\includegraphics[width=.475\textwidth]{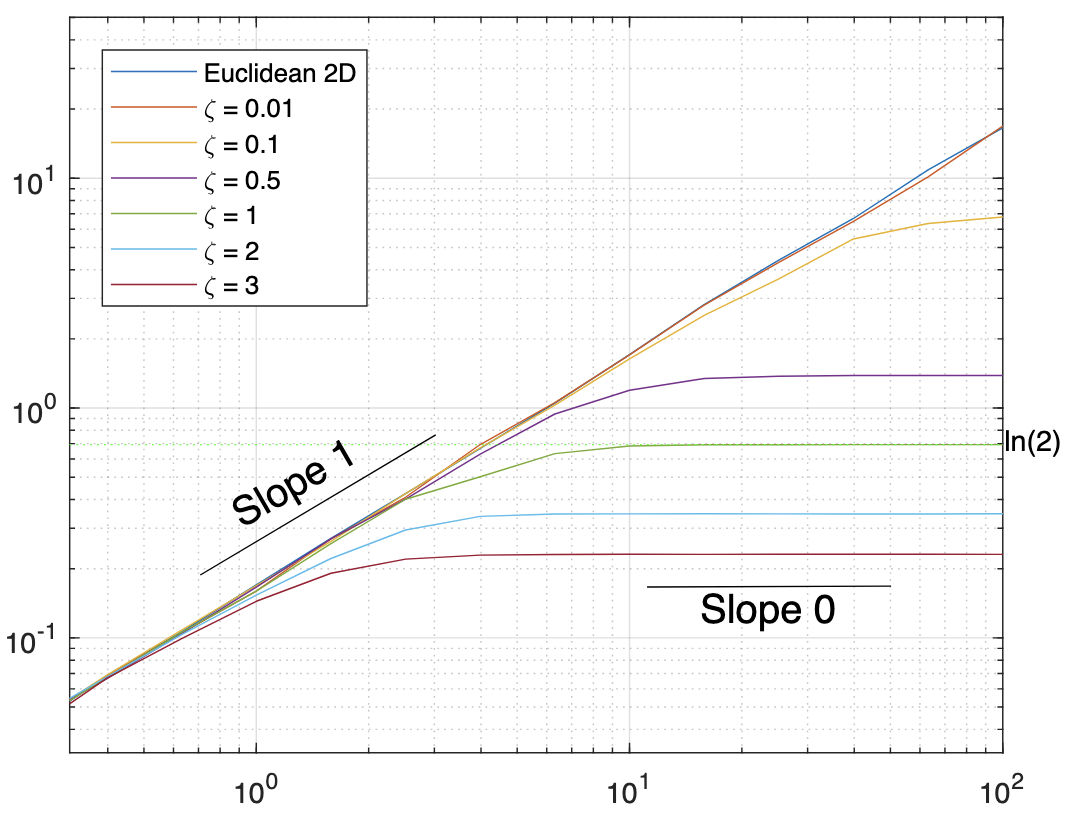} \includegraphics[width=.475\textwidth]{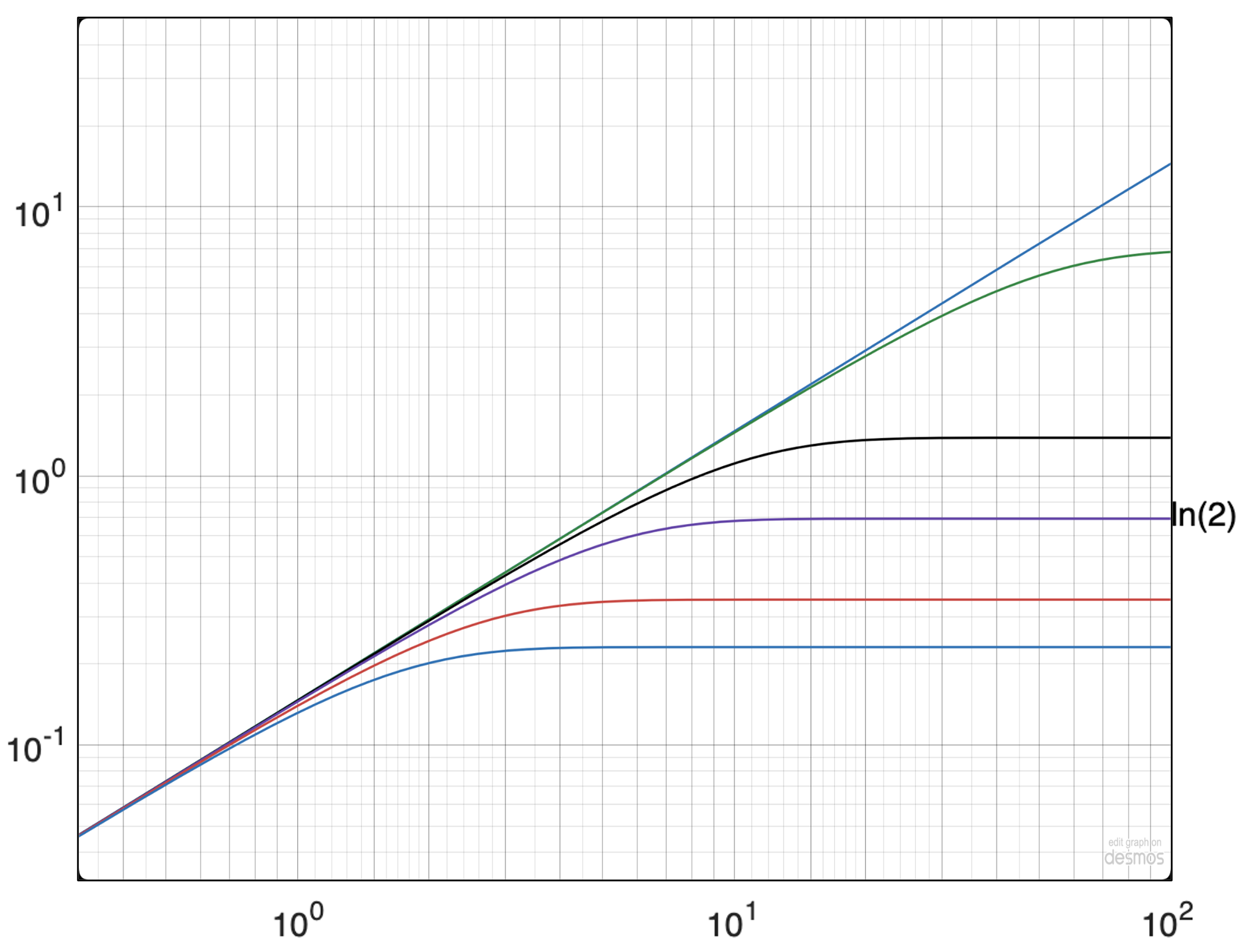}
\caption{The left picture is taken from \cite{Karla2020}. The value of $\zeta$ determines the degree to which the plane is negatively curved. On the right, is a log-log plot of the equation found in our Main Theorem (\cref{ther:Main}). }
 \label{fig:KFig}
\end{figure}
\begin{restatable*}{corollary}{CorA}\label{cor:CorA}
The limit of $\Gamma_{\mathbb{H}^2}$ as $x$ goes to infinity is given by
\[\lim\limits_{x \to \infty} \Gamma_{\mathbb{H}^2}(x) = \ln 2\]
\end{restatable*}
While the theory of $\delta$-hyperbolic metric spaces has existed since the 1980s, a proof of the minimal value of $\delta$ for the hyperbolic plane was only first given in 2016 \cite{Nica2016}. In \cref{sec:3.2}, we use \cref{cor:CorA} to give an alternative proof that the best constant for $\delta$-hyperbolicity on the hyperbolic plane $\mathbb{H}^2$ is $\ln 2$.
\begin{restatable*}{corollary}{CorB}\label{cor:Cor2}
The limit of $\Gamma_{\mathbb{H}^2}'$ as $x$ goes to $0$ is given by
\[\lim\limits_{x \to 0^+} \Gamma_{\mathbb{H}^2}'(x) = \frac{\sqrt{2}-1}{\sqrt{2}}\]
\end{restatable*}
In \cite{Jonckheere2011}'s Major Result, it is shown that 
\[\sup_{\square \subset \mathbb{H}^2} \frac{\delta(\square)}{D(\square)} = \sup_{\square \subset \mathbb{E}^2} \frac{\delta(\square)}{D(\square)} = B\]
for a constant $B$ and one of various definitions of $D(\square)$. They use this result to define a scaled Gromov four-point condition for characterizing hyperbolicity in finite graphs. They were able to give theoretical proofs for the value of $B$ for every definition of $D(\square)$, except when $D(\square)$ is defined as in \cref{def:Def2}. In this case, they give a numerical result. In \cref{sec:3.3}, we use \cref{cor:Cor2} to, when $D(\square)$ is the diameter, give a theoretical proof that $B = (\sqrt{2}-1)/\sqrt{2}$.

\section{Background}

\subsection{The Hyperbolic Plane}

The hyperbolic plane $\mathbb{H}^2$ is the complete, simply connected, surface with constant curvature equal to -1. An explicit model of the hyperbolic plane is the Poincaré disk. The Poincaré disk is the open disk $\{ (x,y) \in \mathbb{R}^2 : \sqrt{x^2 + y^2} < 1\}$ in $\mathbb{R}^2$ with the Riemannian metric $g = 4(dx^2 + dy^2)/(1-x^2-y^2)^2$. The notions of line, line segment, distance, and angle are all defined in terms of this Riemannian metric. Once these definitions are made, the Poincaré disk also satisfies the axioms of neutral geometry as outlined in \cite{Harvey2015}. While needed to rigorously define and derive the geometric properties of the hyperbolic plane, the Poincaré disk, Riemannian geometry, and axiomatic geometry are not needed for our discussion and will not be referenced. Instead, we will reference the following derived geometric properties of the hyperbolic plane:

\begin{itemize}
\item[1.] Consider two points $a,b$ in the hyperbolic plane. There is a unique line passing through $a$ and $b$, denoted as $\overline{ab}$. There is a unique line segment from $a$ to $b$,  denoted as $ab$. The distance between $a$ and $b$ is defined as the length of the line segment $ab$, and is denoted by $|ab|$. 

\item[2.] Three points $a,b,c$ in the hyperbolic plane obey the triangle inequality:
\[|ac| \leq  |ab| + |bc|.\]
Moreover, they obey the stronger hyperbolic law of cosines:
\[
\cosh|ab| = \cosh\left(|ac|\right)\cosh\left(|bc|\right) - \sinh\left(|ac|\right)\sinh\left(|bc|\right)\cos \angle acb\]
where $\angle acb$ is the angle. Applying the hyperbolic trig identity
\[\cosh(x\pm y) = \cosh x \cosh y \pm \sinh x \sinh y.\] 
to the hyperbolic law of cosines gives the following alternate formulation:
\begin{equation} \label{eq:HLaw2}   \cosh|ab| = \cosh\left(|ac|+|bc|\right) \frac{1-\cos \angle acb}{2} + \cosh\left(|ac|-|bc|\right) \frac{1+\cos \angle acb}{2}.
\end{equation}
The points also obey the hyperbolic law of sines:
\[\frac{\sin \angle acb}{\sinh(|ab|)}=\frac{\sin \angle abc}{\sinh(|ac|)}=
\frac{\sin \angle bac}{\sinh(|bc|)}.\]

\item[3.] Four points $a,b,c,d$ in the hyperbolic plane obey the following two theorems:
\begin{theorem}[\cite{Harvey2015} Theorem 2.6] \label{Harvey2.6}
 If $d$ is an interior point of the angle $\angle bac$, then the line $\overline{ad}$ intersects the line segment ${bc}$. 
\end{theorem} 
\begin{theorem}[\cite{Harvey2015} Lemma 2.8] \label{Harvey2.8} 
If $c$ and $d$ are both on the same side of $\overline{ab}$, then either $c$ is an interior point of $\angle abd$ or $d$ is an interior point of $\angle abc$. 
\end{theorem} 
where "on the same side" and "interior point" are given rigorous definitions in \cite{Harvey2015} as Definitions 1.3 and 2.2 respectively. 
\end{itemize}

\subsection{Gromov's $\delta$-Hyperbolic Metric Spaces}

Unlike surfaces, graphs do not have a notion of curvature. However, they do have a notion of distance. In this paper, all graphs are assumed to be connected, weighted, and finite (if a graph is unweighted, we assign all edges the weight one). The distance $|ab|$ between these two points $a,b$ in the graph is defined as the sum of the edge weights in the shortest path between them. This notion of distance gives a graph the structure of a \textit{metric space}.

\begin{definition}
A metric space is a set $X$ equipped with a distance function $d : X \times X \to \mathbb{R}$ that obeys the following conditions: 
\begin{enumerate}
\item The distance $d(a,b)$ between two points $a,b \in X$ is nonnegative, and equal to zero iff $a = b$. \\
\item The distance $d(a,b)$ from $a$ to $b$ is always the same as the distance $d(b,a)$ from $b$ to $a$. \\
\item For three points $a,b,c \in X$, the triangle inequality holds: $d(a,b) \leq d(a,c) + d(c,b)$. 
\end{enumerate}
\end{definition}
Note that the hyperbolic plane $\mathbb{H}^2$ is a metric space, and the Euclidean/Cartesian plane $\mathbb{E}^2$ is also a metric space.

The structure of a metric space allows us to generalize the notion of hyperbolic geometry to that of a $\delta$-hyperbolic metric space: Let $X$ be a metric space. We will refer to four points in a metric space as a quadruple and will sometimes use the notation $\square \subset X$ to denote a quadruple in a metric space $X$. 
\begin{figure}[h]
\centering
\includegraphics[width=.45\textwidth]{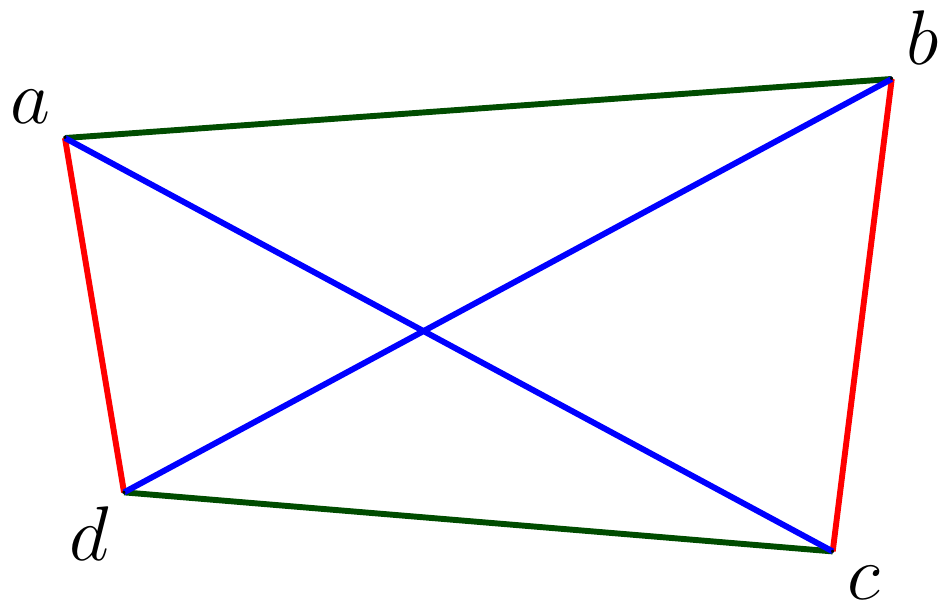} 
\caption{A quadruple with 
S = |ad| + |bc|, M =  |ab| + |cd|, and L = |ac| + |bd|. }
\label{fig:GQuad}
\end{figure}

\begin{definition} \label{def:Def1} Let $\square = \{a,b,c,d\} \subset X$ be a quadruple in a metric space $X$. We define the delta $\delta(\square)$ of $\square$ as follows: There are three ways to divide the four points $a,b,c,d$ into pairs. Take the sums of the distances in each of these pairs: \begin{align}
&d(a,b) + d(c,d) \\
 &d(a,c) + d(b,d) \\
 & d(a,d) + d(b,c).
\end{align} 
We call these sums the pairwise distance sums of $\square$. Then $\delta(\square)$ is half the difference between the two largest sums. In other words, label the sums $S,M,L$ such that $S \leq M \leq L$ (See \cref{fig:GQuad}). Then 
\[\delta(\square) = \frac{L - M}{2}.\]
\end{definition}

Classically, the supremum of $\delta$ over all quadruples characterizes the degree of hyperbolicity in the space \cite{Gromov1987}. 
\begin{definition}
Let $X$ be a metric space and let $\delta \in [0,\infty)$. If $\delta(\square) \leq \delta$ for all $\square \subset X$ (this is called the Gromov four-point condition), then $X$ is called a $\delta$-hyperbolic metric space. 
\end{definition}

\begin{figure}[h]
\centering
\includegraphics[width=.75\textwidth]{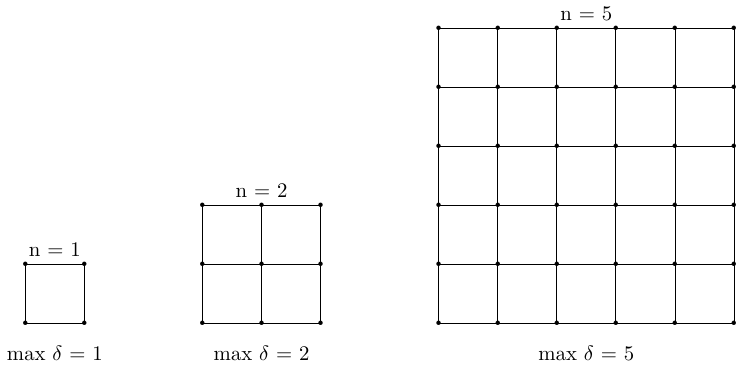}
\caption{Lattice graphs for $n = 1,2,5$.}
\label{fig:Lattice}
\end{figure}

Since connected graphs are metric spaces with the shortest path metric, the notion of $\delta$-hyperbolicity seemingly gives us a function from the set of graphs to the real line that describes the global hyperbolicity of the graph. However, $\delta$-hyperbolicity is designed primarily for metric spaces where the distance is unbounded. For finite graphs, which have bounds on their distance, $\delta$-hyperbolicity is not a well-defined characterization. For example, a lattice graph for some $n$ will have $\delta = n$ (See \cref{fig:Lattice}) \cite{chen2013hyperbolicitysmallworldtreelikerandom}. However, the lattice graphs are known to have quasi-Euclidean geometry and the Euclidean plane has unbounded $\delta$. We want a characterization that can identify the lattice graph as Euclidean independent of $n$. 

The limitation comes from applying $\delta$, which considers the maximum over all quadruples, to metric spaces with finite diameter. In this case, $\delta$ depends primarily on the diameter of the metric space and not on its geometry.  In order to overcome this hurdle, we need to consider not only the $\delta$ of a quadruple, but also its diameter. 
\begin{definition} \label{def:Def2} Let $X$ be a metric space. The diameter $D(\square)$ of a quadruple $\square \subset X$ is defined as the supremum of distances between any two of its points: $D(\square) = \sup\{d(a,b) : a,b \in \square\}$.
\end{definition}

\section{$\Gamma_X$ on Metric Spaces}

Our primary tool for studying the hyperbolicity of a metric space $X$ will be a function defined on the set of possible distances in $X$. This function, called $\Gamma_X$, will measure how large $\delta(\square)$ can be when $\square$ has diameter $x$, for every $x \in \text{Dist}_X$. 

\GDef
\begin{proposition}\label{prop:maxg}
For a metric space $X$, we have 
\[\Gamma_{X}(x) \leq \frac{x}{2}.\]
\end{proposition} 
\begin{proof}
Consider a quadruple $\square$ labeled $a,b,c,d$ such that $D(\square) = x$. Suppose, without loss of generality, that
\[S = d(a,b) + d(c,d) \leq M = d(a,c) + d(b,d) \leq L = d(a,d) + d(b,c),\] see \cref{def:Def1}. By the triangle inequality, we have 
\[\begin{aligned} 
d(a,d) \leq  d(a,b) + d(b,d), \\
d(a,d) \leq  d(c,d) + d(a,c), 
\end{aligned} \quad 
 \begin{aligned} 
d(b,c) \leq d(a,b) + d(a,c), \\
d(b,c) \leq d(c,d) + d(b,d). \\
\end{aligned}
\]
This implies that
\[\begin{alignedat}{1}
2L &= 2d(a,d) + 2d(b,c)  \\
&\leq  |ab| + |bd| +  |cd| + |ac| + |ab| + |ac| + |cd| + |bd|  \\
&= 2(M+S) \\
&\leq 4M .\\
\end{alignedat}\]
Recall that the diameter $x$ is the largest pairwise distance in $\square$. Then $L = d(a,d) + d(b,c) \leq 2x$. It follows that 
\[2L - 2M \leq 2L - L = L \leq 2x\]
which implies  $\delta(\square) = (L-M)/2 \leq x/2$. 
\end{proof}

\begin{proposition} \label{prop:0quad}
Let $X$ be a metric space and let $a,b,c,d \in X$ such that $\delta(a,b,c,d) \neq 0$. Then $a,b,c,d$ are distinct points in $X$. 
\end{proposition}
\begin{proof}
Suppose without loss of generality that $a = b$. Then by the triangle inequality, \[|ab| + |cd| = |aa| + |cd| = |cd| \leq |ac| + |ad| = |ac| + |bd| =  |ad| + |bc|.\]
It follows that $S = |ab| + |cd|$ and $M = L = |ac| + |bd| =  |ad| + |bc|$. Which implies that $\delta(\{a,b,c,d\}) = (L-M)/2 = 0$, proving the contrapositive. 
\end{proof}

\section{The hyperbolic plane}

In this section we will study the behavior of $\Gamma_{\mathbb{H}^2}$. We will begin with some preliminary concepts. Then we will move on to the proof of our main theorem. We will conclude with an alternate proof of the best constant of $\delta$-hyperbolicity for the hyperbolic plane $\mathbb{H}^2$ and find the explicit upper bound for the scaled Gromov four-point condition in the diameter scaling case. 

\begin{figure}[h]
  \begin{subfigure}{0.45\textwidth}
    \includegraphics[width=\linewidth]{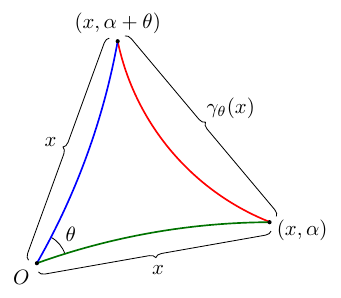}
\caption{} \label{fig:gfunction}
  \end{subfigure}
  \hspace*{\fill}   
  \begin{subfigure}{0.45\textwidth}
    \includegraphics[width=\linewidth]{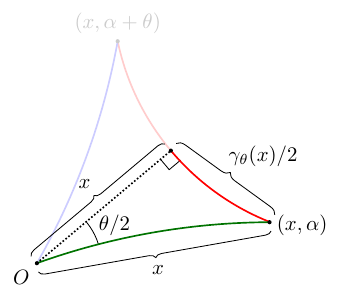}
\caption{} \label{fig:gfunction2}
  \end{subfigure}
 \caption{
The triangle in \cref{fig:gfunction2} is obtained by taking the angle bisector of $\theta$ in the isosceles triangle shown in \cref{fig:gfunction} (see Sections~4.8 (p.~44) and~7.12 (p.~71) of \cite{Harvey2015}).
} 
\label{fig:Gfunction}
\end{figure}

We may define polar coordinates on $\mathbb{H}^2$ where the radius is the distance from some reference point $O$ (the origin), and the angle is determined relative to some ray emanating from $O$. Because the hyperbolic plane is strongly isotropic, the choice of $O$ and the emanating ray can be made arbitrarily. We first define a function that measures the distance between two points that are equidistant from the origin, in terms of that distance and the angle the two points make with the origin:
\begin{definition}
For $\theta \in [0,2\pi)$, define $\gamma_\theta : [0,\infty) \to \mathbb{R}$ by $\gamma_\theta(x) = |(x,\alpha), (x,\alpha + \theta)|$
where $(x,\alpha)$ and $ (x,\alpha + \theta)$ are given in polar coordinates (see \cref{fig:Gfunction}).
\end{definition}
We will use this function throughout our study of $\Gamma$ on the hyperbolic plane. By applying the hyperbolic law of cosines and our alternative form (see \cref{eq:HLaw2}), to the triangle in \cref{fig:gfunction}, we have
\begin{align}
\gamma_\theta(x) &=\label{eq:GLaw1} \cosh^{-1}\left( \cosh^2\left(x\right)-\sinh^2\left(x\right)\cos(\theta)\right) \\
&=\label{eq:GLaw2}  \cosh^{-1}\left(\cosh\left(2x\right)\frac{1-\cos(\theta)}{2}
+ \frac{1+\cos(\theta)}{2}\right) .
\end{align}
By applying the hyperbolic law of sines to the triangle in \cref{fig:gfunction2}, we have
\begin{equation}\label{eq:GLaw3} 
\gamma_\theta(x) = 2\sinh^{-1}\left(\sin\left(\frac{\theta}{2}\right)\sinh\left(x\right)\right)
\end{equation}
Note that $\gamma_\theta$ is continuous on its domain $[0,\infty)$ and smooth on $(0,\infty)$.  We will frequently use the following two technical lemmas on $\gamma_\theta$:
\begin{restatable}{lemma}{TLemA} \label{lem:TLemma1}
$\gamma'_\theta(x) \geq 0$ and $\gamma''_\theta(x) \geq 0$ for $x \in (0,\infty)$. 
\end{restatable}
\begin{proof}
Let $S=\sin(\theta/2)$. Then differentiating \cref{eq:GLaw3} gives us 
\[
\gamma_\theta'(x)=\frac{2S\cosh(x)}{\sqrt{1+S^2\sinh^2(x)}} \quad
\gamma_\theta''(x)=
\frac{2S(1-S^2)\sinh(x)}{\left(1+S^2\sinh^2(x)\right)^{3/2}}.
\]

Since $\theta \in [0,2\pi)$, it follows that $\sin(\theta/2) \geq 0$. Moreover, $\cosh(x)$ is always greater than $0$. It follows that $\gamma_\theta'(x) \geq 0$ for $x \in (0,\infty)$. Since $1-\sin^2(\theta/2) \geq 0$, and $\sinh(x)$ is greater than $0$ for $x \in (0,\infty)$, it follows that $\gamma_\theta''(x) \geq 0$.
\end{proof}

\begin{restatable}{lemma}{TLemB}  \label{lem:TLemma2}
$\lim\limits_{x\to 0^+} \gamma'_{\pi/2}(x) = \sqrt{2}$ and $\lim\limits_{x\to \infty} \gamma'_{\pi/2}(x) = 2$. 
\end{restatable}
\begin{proof}
Plugging $\theta = \pi/2$ into $\gamma_\theta'(x)$ and dividing numerator and denominator by $\cosh(x)$ gives us 
\[\gamma_{\pi/2}'(x)=\frac{\sqrt{2}\cosh(x)}{\sqrt{1+\frac{\sinh^2(x)}{2}}} =
\frac{\sqrt{2}}{\sqrt{\left(\frac{1}{\cosh(x)}\right)^2+\frac{\tanh^2(x)}{2}}}  \]
Note that $\lim\limits_{x\to 0^+} \cosh(x) = 1$ and $\lim\limits_{x\to 0^+} \sinh(x) = 0$. It follows from the first equality that $\lim\limits_{x\to 0^+} \gamma'_{\pi/2}(x) = \sqrt{2}$. Note that  $\lim\limits_{x\to \infty} 1/\cosh(x) = 0$ and $\lim\limits_{x\to \infty} \tanh(x) = 1$. It follows from the second equality that $\lim\limits_{x\to \infty} \gamma'_{\pi/2}(x) = \sqrt{2}/\sqrt{1/2} = 2$. 
\end{proof}

\subsection{Main Theorem}

\begin{figure}[h]
\centering
\includegraphics[width=0.65\textwidth]{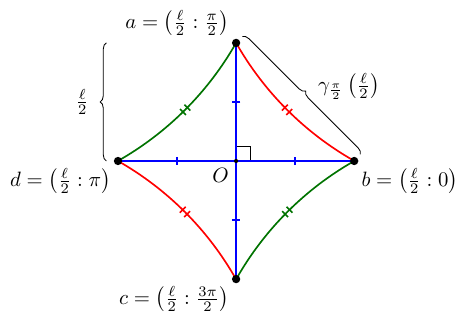}
\caption{The quadruple $\square_\ell$ in the hyperbolic plane. Note that $|ac| = |bd| = \ell$ and $|ab| = |bc| = |cd| = |ad| = \gamma_{\pi/2}(\ell/2)$. It follows from the triangle inequality that
$\gamma_{\pi/2}(\ell/2) = |ab| \leq |Oa| + |Ob| = \ell/2 + \ell/2 = \ell$.
Therefore $D(\square_\ell) = \ell$,
 $L(\square_{\ell}) = 2\ell$, and 
 $M(\square_{\ell}) = 2\gamma_{\pi/2}(\ell/2)$.
 }
\label{fig:LQuad}
\end{figure}

This section will be dedicated to the proof of our main theorem:
\Main
\begin{proof}
Define $f : [0,\infty) \to \mathbb{R}$ by $f(x) = x - \cosh^{-1}\left(\cosh^2(x/2)\right)$ and note that 
\[f(x)  = x - \gamma_{\pi/2}\left(\frac{x}{2}\right) = \frac{2x - 2 \gamma_{\pi/2}\left(\frac{x}{2}\right)}{2} = \frac{L(\square_x) - M(\square_x)}{2} = \delta(\square_x)\]
where the quadruple $\square_x$ with diameter $D(\square_x) = x$ is given by setting $\ell = x$ in the family
\[
\square_\ell=
\left\{
\left(\frac{\ell}{2},0\right),
\left(\frac{\ell}{2},\frac{\pi}{2}\right),
\left(\frac{\ell}{2},\pi\right),
\left(\frac{\ell}{2},\frac{3\pi}{2}\right)
\right\},
\qquad \ell\geq 0,
\]
see \cref{fig:LQuad}. As in the definition of $\gamma_\theta$, $\square_\ell$ is given in polar coordinates. 

We will complete the proof by showing that $f'(x) \geq 0$, $f''(x) \leq 0$ for $x \in (0,\infty)$, and that $\Gamma_{\mathbb{H}^2}(x) = f(x)$ for $x \in [0,\infty)$. 

Define $g : (0,\infty) \to \mathbb{R}$ by $g(x) = \gamma_{\pi/2}(x/2)$, so that $f(x) = x - g(x)$. By \cref{lem:TLemma1}, $g''(x) = \gamma_{\pi/2}''(x/2)/4 \geq 0$. So $g'(x)$ is increasing and therefore $g'(x) \leq \lim_{x \to \infty} g'(x)$. By \cref{lem:TLemma2}, $\lim_{x \to \infty} \gamma'_{\pi/2}(x) = 2$. Thus, 
\[g'(x) \leq \lim_{x \to \infty} g'(x) = \frac{1}{2} \lim_{x \to \infty} \gamma_{\pi/2}'\left(\frac{x}{2}\right) = \frac{2}{2} = 1.\]
It follows that $f'(x) = 1 - g'(x) \geq 0$. Moreover, by 
\cref{lem:TLemma1}, \[f''(x) = -g''(x) = -\frac{\gamma''_{\pi/2}\left(\frac{x}{2}\right)}{4} \leq 0.\] 

Recall that $\Gamma_{\mathbb{H}^2}(x)$ is the supremum of $\delta(\square)$ over all quadruples $\square$ with diameter $x$. By \cref{lem:PLem1}, we may restrict to the subset $S$ of quadruples satisfying condition 1 of \cref{lem:PLem1}. By \cref{lem:MLemma}, every quadruple $\square \in S$ satisfies $\delta(\square) \leq \delta(\square_{L(\square)/2}) = f(L(\square)/2)$. Thus, 
\[\Gamma_{\mathbb{H}^2}(x) = \sup_{\square \in S} \delta(\square) \leq \sup_{\square \in S} f\left(\frac{L(\square)}{2}\right).\]
The sum of disjoint distances $L(\square)$ is bounded by twice the longest distance $D(\square)$. Since $\square \in S$ have their diameter bounded by $x$, the maximal value of $L(\square)/2$ in the range of the supremum is $2x/2 = x$. Note that $\square_x$ is an example of a quadruple in $S$ with diameter $x$ (see \cref{fig:LQuad}). 
Since $f$ is increasing, it follows that 
\[\Gamma_{\mathbb{H}^2}(x) = \sup_{\square \in S} \delta(\square) = \sup_{\square \in S} f\left(\frac{L(\square)}{2}\right) = f(x).\]

\begin{figure}[h]
  \begin{subfigure}{0.225\textwidth}
    \includegraphics[width=\linewidth]{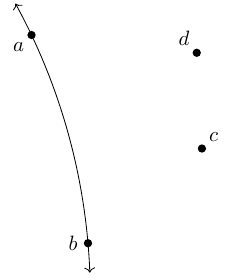}
\caption{} \label{fig:Lemma3.1}
  \end{subfigure}
  \hspace*{\fill}   
  \begin{subfigure}{0.225\textwidth}
    \includegraphics[width=\linewidth]{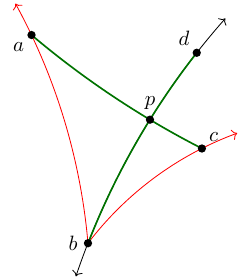}
\caption{} \label{fig:Lemma3.2}
  \end{subfigure}
  \hspace*{\fill}   
  \begin{subfigure}{0.225\textwidth}
    \includegraphics[width=\linewidth]{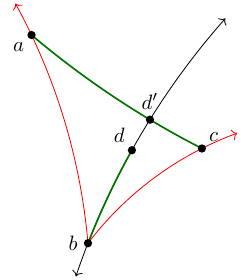}
\caption{} \label{fig:Lemma3.3}
  \end{subfigure}
  \hspace*{\fill}   
\caption{
\textbf{(a)} $c$ and $d$ lie on the same side of $\overline{ab}$ and $d$ lies in the interior of $\angle abc$. 
\textbf{(b)} Case 1: $d$ does not lie on the same side of $p$ as $b$, therefore $ac$ and $bd$ intersect. 
\textbf{(c)} Case 2: $d$ lies on the same side of $p$ as $b$, therefore $ac$ and $bd$ do not intersect. However, $\overline{bd}$ intersects $ac$ at a point $d'$.  
} 
\end{figure}
\end{proof}
We will now prove \cref{lem:PLem1,lem:MLemma}. 
 \begin{lemma}\label{lem:PLem1}
For any quadruple $\square \subset \mathbb{H}^2$ with diameter $\ell$, there exists a quadruple $\square' \subset \mathbb{H}^2$ with diameter $\ell$ satisfying the following conditions:
\begin{enumerate}
    \item  A pair of line segments joining two disjoint pairs of vertices in $\square'$ intersect. In other words, $\square'$ can be labeled $\{a,b,c,d\}$ such that $ac$ and $bd$ intersect. 
    \item  $\delta(\square) \leq \delta(\square').$
\end{enumerate}
\end{lemma}
\begin{proof}

Consider $\square \subset \mathbb{H}^2$ with diameter $\ell$. We first label $\square$ by $\{a,b,c,d\}$. If $a,b$ lie on different sides of the line $\overline{cd}$, and $c,d$ lie on different sides of the line $\overline{ab}$, then $ab$ and $cd$ intersect. Thus, we may set $\square' = \square$ and we are done. 

Therefore, we assume without loss of generality that $c$ and $d$ lie on the same side of the line $\overline{ab}$ as in \cref{fig:Lemma3.1}. By \cref{Harvey2.8} either $c$ is an interior point of $\angle abd$ or $d$ is an interior point of $\angle abc$. We can permute the labels of $c$ and $d$ without affecting our initial assumption. We may then, without loss of generality, assume that $d$ is an interior point of $\angle abc$. By \cref{Harvey2.6}, $\overline{bd}$ intersects $ac$ at a point $p$. 

If $b$ and $d$ do not lie on the same side of $p$, then $p$ must lie on $bd$ as in \cref{fig:Lemma3.2}. Thus, $p$ lies on both $ac$ and $bd$. Then $ac$ and $bd$ intersect, so we may set $\square' = \square$ and we are done. Suppose then that $b$ and $d$ lie on the same side of $p$ as in \cref{fig:Lemma3.3}. Then, since $d$ is an interior point of $\angle abc$, $d$ lies in the interior of the triangle $\triangle abc$. We may now permute the labels $a,b,c$ so that $L(\square) = |ac| + |bd|$. Since $d$ still lies in the interior of $\triangle abc$, it is still an interior point of the angle $\angle abc$. By \cref{Harvey2.6}, the line $bd$ intersects the line segment $ac$. Let $d'$ be the point of this intersection and set $\square' = \{a,b,c,d'\}$. Note that, since $d$ lies inside $\triangle abc$, the diameter of $\square'$ is unchanged from $\ell$. We conclude the proof by showing that $\delta(\square) \leq \delta(\square')$. 

Let
\[
S_1=|ab|+|cd|,\quad
S_2=|ac|+|bd|,\quad
S_3=|ad|+|bc|
\]
be the pairwise distance sums (see \cref{def:Def1}) of $\square$ and likewise define $S_1',S_2',S_3'$ for $\square'$. Let $\Delta = |dd'|$.  By the triangle inequality, $|cd'| \leq |cd| + |dd'| = |cd| + \Delta$. Then 
$|cd'| - |cd| \leq \Delta$. It follows that 
\[S_1' - S_1 = |ab| + |cd'| - |ab| - |cd| = |cd'| - |cd| \leq \Delta.\]
It similarly follows that 
\[S_3' - S_3 \leq \Delta.\] Meanwhile, we have \[S_2' - S_2 = |ac| + |bd'| - |ac| - |bd| = |bd'| - |bd| = |dd'| = \Delta.\]
Since $S_2 = L(\square)$, $S_1, S_3 \leq S_2$ and $S_1' - S_1, S_3' - S_3 \leq S_2' - S_2 = \Delta$ imply that $L(\square') = S_2'$, and thus
\[L(\square') = S_2' = S_2 + \Delta = L(\square)+\Delta.\]
Then, $M(\square')$ must be the larger of $S_1'$ and $S_3'$. We have \[M(\square') = \text{max}\{S_1',S_3'\} 
\leq \text{max}\{S_1+\Delta,S_3+\Delta\} \leq \text{max}\{S_1,S_3\} +\Delta = M(\square) + \Delta\]
It follows that $M(\square') - M(\square) \leq L(\square') - L(\square)$, which implies that \[\delta(\square) = \frac{L(\square) -M(\square)}{2} \leq  \frac{L(\square') -M(\square')}{2} = \delta(\square').\]

\end{proof}

\begin{lemma}\label{lem:MLemma}
For any quadruple $\square \subset \mathbb{H}^2$ satisfying condition 1 of \cref{lem:PLem1}, we have \[\delta(\square) \leq \delta(\square_{L(\square)/2}) \]
where $\square_{L(\square)/2}$ is given by setting $\ell = L(\square)/2$ in the family $\square_\ell$ (see \cref{fig:LQuad}). 
\end{lemma}
\begin{proof}

Consider $\square \subset \mathbb{H}^2$ satisfying condition 1 of \cref{lem:PLem1}. Then $\square$ can be labeled $\{a,b,c,d\}$ such that $ac$ and $bd$ intersect at a point $p$. Let $p$ be the origin. Since $a$ and $c$ lie on the same line through the origin, as well as $b$ and $d$, we can give $\{a,b,c,d\}$ the polar coordinates
\[(r_a, \alpha), (r_b, \beta), (r_c, \alpha+\pi), (r_d, \beta + \pi).\]
It follows from the triangle inequality that $|ab|  \leq |ap| + |pb|$ and $|cd| \leq  |cp| + |pd|$. Since $|ac| = |ap| + |cp|$ and $|bd| = |pb| + |pd|$, we have 
\[|ab| + |cd| \leq  |ap| + |pb| + |cp| + |pd| = |ac| + |bd|.\]
It likewise follows that $|ad| + |bc| \leq |ac| + |bd|$. So $L(\square) = |ac| + |bd|$ and we have 
\begin{equation} \label{eq:RSum}
L(\square) = r_a + r_b + r_c + r_d.
\end{equation}
Note that $\angle apb + \angle apd = \pi$, so either $\angle apb \geq \pi/2$ or $\angle apd \geq \pi/2$. We may permute the labels of $b$ and $d$ without changing the fact that $ac$ intersects $bd$. So we may, without loss of generality, assume that 
 \begin{equation} \label{eq:BigAngle}
\angle apb = |\alpha - \beta| \geq \frac{\pi}{2}.
\end{equation}

By the definition of $\square_\ell$, we have $L(\square_{L(\square)/2})  = L(\square)$ (see \cref{fig:LQuad}). Thus, it remains to show that $M(\square) \geq M(\square_{L(\square)/2})$. We will show this via a chain of starred inequalities. Since $L(\square) = |ac| + |bd|$, it follows that
\begin{equation}\label{eq:eq1}
M(\square) \geq  |ab| + |cd|.
\end{equation}
We will now show that $|ab| \geq \gamma_{|\alpha - \beta|}((r_a+r_b)/2)$.
By our alternative form for the hyperbolic law of cosines, \cref{eq:HLaw2}, we have
\begin{align*}
|ab| &= \left|\left(r_a,\alpha\right),\left(r_b,\beta\right)\right| \\
&= \cosh^{-1}\left(\cosh \left(r_a+r_b\right) \frac{1-\cos |\alpha - \beta|}{2}  + \cosh \left( r_a-r_b \right) \frac{1+\cos |\alpha - \beta|}{2} \right).
\end{align*}
Because $\cosh(x) \geq 1$ and $\cosh^{-1}(x)$ is increasing for positive $x$, 
it follows that
\begin{align*}
& \cosh^{-1}\left(\cosh \left(r_a+r_b\right) \frac{1-\cos |\alpha - \beta|}{2}  + \cosh \left( r_a-r_b \right) \frac{1+\cos |\alpha - \beta|}{2} \right) \\
\geq{}& \cosh^{-1}\left(\cosh \left(r_a+r_b\right) \frac{1-\cos |\alpha - \beta|}{2} +  \frac{1+\cos |\alpha - \beta|}{2} \right)
 .
\end{align*}
By \cref{eq:GLaw2}, 
\[ \gamma_{|\alpha - \beta|}\left(\frac{r_a+r_b}{2}\right) = \cosh^{-1} \left(\cosh \left(r_a+r_b\right) \frac{1-\cos |\alpha - \beta|}{2} +  \frac{1+\cos |\alpha - \beta|}{2}\right) .\]
Thus $|ab|\geq \gamma_{|\alpha - \beta|}((r_a+r_b)/2)  $. We likewise have $ |cd| \geq \gamma_{|\alpha - \beta|}((r_c+r_d)/2) $. So
\begin{equation}\label{eq:eq2}
 |ab| + |cd| \geq \gamma_{|\alpha - \beta|}\left(\frac{r_a+r_b}{2}\right) + \gamma_{|\alpha - \beta|}\left(\frac{r_c+r_d}{2}\right).
\end{equation}
Recall from \cref{lem:TLemma1} that $\gamma_{|\alpha - \beta|}'(x), \gamma_{|\alpha - \beta|}''(x) \geq 0$ for $x \in (0,\infty)$. Thus, since $\gamma_{|\alpha - \beta|}$ is continuous on $\mathbb{R}$, it follows that $\gamma_{|\alpha - \beta|}$ is non-decreasing and convex on $[0,\infty)$. Therefore
\[\gamma_{|\alpha - \beta|}\left(\frac{\frac{r_a+r_b}{2} + \frac{r_c+r_d}{2} }{2}\right) \leq \frac{\gamma_{|\alpha - \beta|}(\frac{r_a+r_b}{2}) +  \gamma_{|\alpha - \beta|}(\frac{r_c+r_d}{2})}{2}.\]
By \cref{eq:RSum}, we have $r_a+r_b+r_c+r_d = L(\square)$. It follows that
\begin{equation}\label{eq:eq3}
\gamma_{|\alpha - \beta|}\left(\frac{r_a+r_b}{2}\right) +  \gamma_{|\alpha - \beta|}\left(\frac{r_c+r_d}{2}\right) \geq 2\gamma_{|\alpha-\beta|}\left(\frac{r_a+r_b+r_c+r_d}{4}\right) = 2\gamma_{|\alpha-\beta|}\left(\frac{L(\square)}{4}\right).
\end{equation}
We will now show that $2\gamma_{|\alpha-\beta|}(L(\square)/4) \geq M(\square_{L(\square)/2})$. 
By the definition of $\square_\ell$, we have $M(\square_{L(\square)/2}) = 2\gamma_{\pi/2}(L(\square)/4)$ (see \cref{fig:LQuad}). It follows that 
\[M(\square_{L(\square)/2}) = 2\gamma_{\pi/2}\left(\frac{L(\square)}{4}\right) = 2\cosh^{-1}\left(\cosh^2\left(\frac{L(\square)}{4}\right)\right).\]
By \cref{eq:BigAngle}, we have $|\alpha-\beta | \geq \pi/2$. It follows that $\cos |\alpha - \beta| \leq 0$. Thus, since $\cosh^{-1}(x)$ is increasing for positive $x$, we have 
\[
2\cosh^{-1}\left(\cosh^2\left(\frac{L(\square)}{4}\right)\right)  \leq 2\cosh^{-1}\left(\cosh^2\left(\frac{L(\square)}{4}\right) - \sinh^2\left( \frac{L(\square)}{4}\right) \cos |\alpha - \beta|\right) 
\]
By \cref{eq:GLaw1}, we have 
\[2\gamma_{|\alpha-\beta|}\left(\frac{L(\square)}{4}\right) = 2\cosh^{-1}\left(\cosh^2\left(\frac{L(\square)}{4}\right) - \sinh^2\left( \frac{L(\square)}{4}\right) \cos |\alpha - \beta|\right).\]
Thus, 
\begin{equation}\label{eq:eq4}
2\gamma_{|\alpha-\beta|}\left(\frac{L(\square)}{4}\right) \geq 2\cosh^{-1}\left(\cosh^2\left(\frac{L(\square)}{4}\right)\right)  = M(\square_{L(\square)/2}) .
\end{equation}
\cref{eq:eq1,eq:eq2,eq:eq3,eq:eq4} imply that $M(\square) \geq M(\square_{L(\square)/2})$, which completes the proof.

\end{proof}

\subsection{Gromov's $\delta$-Hyperbolic Space}\label{sec:3.2}

It was first shown in \cite{Nica2016} that $\delta = \ln 2$ is the minimum $\delta$ for which $\mathbb{H}^2$ is $\delta$-hyperbolic in the sense of Gromov. We give an alternative proof using $\Gamma_{\mathbb{H}^2}$:
It follows from \cref{ther:Main} that $\Gamma_{\mathbb{H}^2}$ is non-decreasing, so the limit of $\Gamma_{\mathbb{H}^2}$ as $x$ approaches infinity gives the supremum of its possible values. Then \cref{cor:CorA} implies that the minimum $\delta$ for which $\mathbb{H}^2$ is $\delta$-hyperbolic is given by
\[\delta = \sup\limits_{\square \subset \mathbb{H}^2} \delta(\square) =  \sup\limits_{x \in [0,\infty)} \sup\limits_{\square \subset \mathbb{H}^2 : D(\square) = x} \delta(\square) = \sup\limits_{x \in [0,\infty)} \Gamma_{\mathbb{H}^2}(x) = \lim\limits_{x \to \infty} \Gamma_{\mathbb{H}^2}(x) = \ln 2.\]

\CorA
\begin{proof}

Let $y = \cosh^{-1}\left(\cosh^2(x/2)\right)$. Applying the exponential definition of $\cosh$ to $\cosh (y) = \cosh^2(x/2)$ gives us
\[
e^{y} + e^{-y} = \frac{(e^{x/2} +  e^{-x/2} )^2}{2} 
= \frac{e^x + 2 + e^{-x}}{2} \]
and rearranging this equation gives
\[\begin{alignedat}{1}
y &=  \ln \frac{e^x + 2 + e^{-x} - 2e^{-y} }{2} \\
&=  \ln \frac{(e^{x})(1 + 2e^{-x} + e^{-2x} - 2e^{-(x+y)}) }{2} \\
&= x + \ln(1 + 2e^{-x} + e^{-2x} - 2e^{-(x+y)}) - \ln (2)
\end{alignedat}\]
By \cref{ther:Main} we have 
\begin{equation} 
\Gamma_{\mathbb{H}^2}(x) = x-y = \ln 2 - \ln(1 + 2e^{-x} + e^{-2x} - 2e^{-(x+y)}). 
\end{equation}
Since $y$ is always nonnegative (as $\cosh^{-1}$ is nonnegative), it follows that $e^{-x}, e^{-2x}, e^{-(x+y)}$ will converge to $0$. Therefore, 
\[
\lim\limits_{x \to \infty} \Gamma_{\mathbb{H}^2}(x) =  \ln 2 -  \lim\limits_{x \to \infty} \ln(1 + 2e^{-x} + e^{-2x} - 2e^{-(x+y)}) =  \ln 2 -  \lim\limits_{x \to \infty} \ln(1) = \ln 2.
\]

\end{proof}

\subsection{Scaled Gromov Four-Point Condition}\label{sec:3.3}

A framework of characterizing hyperbolicity in graphs was developed in \cite{Jonckheere2011} by taking the supremum over all quadruples $\square$ of $\delta(\square)$ rescaled by a scaling factor. One of the scaling factors they considered was the diameter (referred to as $D(\square)$ in our paper and $diam(\square)$ in \cite{Jonckheere2011}). The supremum in this case is $\sup_{\square \subset \mathbb{H}^2} \delta(\square)/D(\square)$. They give numerical but not theoretical results for the value of $\sup_{\square \subset \mathbb{H}^2} \delta(\square)/D(\square)$. We give a proof of this value using $\Gamma_{\mathbb{H}^2}$: It suffices to compute $\sup_{\square \subset \mathbb{H}^2} \delta(\square)/D(\square) = \sup_{x \in (0,\infty)} \Gamma_{\mathbb{H}^2}(x)/x$. Since $\Gamma_{\mathbb{H}^2}$ is concave on $(0,\infty)$ by \cref{ther:Main} and continuous at $0$, it is concave on its domain. Since we also have $\Gamma_{\mathbb{H}^2}(0) = 0$, it follows that $\Gamma_{\mathbb{H}^2}(tx) \geq t\Gamma_{\mathbb{H}^2}(x)$ for $t \in [0,1]$. Consider $x,x' \in (0,\infty)$ such that $x \leq x'$. Then $x/x' \in [0,1]$ so that 
\[\Gamma_{\mathbb{H}^2}(x) = \Gamma_{\mathbb{H}^2}\left(\frac{x}{x'}x'\right) \geq \frac{x}{x'} \Gamma_{\mathbb{H}^2}\left(x'\right).\]
Which implies that $\Gamma_{\mathbb{H}^2}(x)/x \geq \Gamma_{\mathbb{H}^2}(x')/x'$. Thus, $\Gamma_{\mathbb{H}^2}(x)/x$ is non-increasing. So the supremum $\sup_{x \in (0,\infty)} \Gamma_{\mathbb{H}^2}(x)/x$ may be obtained by taking the limit as $x$ approaches $0$:
\[\sup_{\square \subset \mathbb{H}^2} \delta(\square)/D(\square) = \sup_{x \in (0,\infty)} \Gamma_{\mathbb{H}^2}(x)/x = \lim_{x \to 0^+} \Gamma_{\mathbb{H}^2}(x)/x.\]
By L'H\^{o}pital's rule, and then applying \cref{cor:Cor2}, we have
\[\sup_{\square \subset \mathbb{H}^2} \delta(\square)/D(\square) = \lim_{x \to 0^+} \Gamma_{\mathbb{H}^2}(x)/x = \lim_{x \to 0^+} \Gamma_{\mathbb{H}^2}'(x)/1 = \frac{\sqrt{2}-1}{\sqrt{2}}.\]

\CorB

\begin{proof}
By \cref{ther:Main} and \cref{eq:GLaw1} we have \[\Gamma_{\mathbb{H}^2}(x)  = x - \cosh^{-1}\left(\cosh^2\left(\frac{x}{2}\right)\right) = x - \gamma_{\pi/2}\left(\frac{x}{2}\right).\]
Recall from \cref{lem:TLemma2} that $\lim\limits_{x \to 0^+} \gamma'_{\pi/2}(x) = \sqrt{2}$. It follows that
\[\lim\limits_{x \to 0^+} \Gamma_{\mathbb{H}^2}'(x)  = \lim\limits_{x \to 0^+} \left(1 - \frac{\gamma'_{\pi/2}\left(\frac{x}{2}\right)}{2} \right)=  1 - \lim\limits_{x \to 0^+} \frac{\gamma'_{\pi/2}\left(\frac{x}{2}\right)}{2} = 1 - \frac{\sqrt{2}}{2} = \frac{\sqrt{2}-1}{\sqrt{2}} \]
\end{proof}

\section{Graphs}
We will now consider $\Gamma_G$ when $G$ is a graph (connected, weighted, finite, and equipped with the path metric). We will first examine the behavior of $\Gamma_G$ on graphs that possess maximally hyperbolic and minimally hyperbolic (Euclidean) properties. We will then develop a method of characterizing the hyperbolicity of a graph.

\subsection{Model Graphs}\label{sec:mod}

Let $T$ be a tree (a graph with no cycles). It is generally understood that trees can be considered as models of maximal hyperbolicity \cite{Gromov1987}. Moreover, it is also known that $\delta(\square)$ vanishes for any quadruples $\square \subset T$ \cite{Gromov1987}. Then $\Gamma_{T}(x) = 0$ for all $x \in \text{Dist}_T$.

We now consider the lattice graph $L_n$ (see \cref{fig:Lattice}). It is generally understood that lattice graphs model quasi-Euclidean geometry. Let $L_n$ be the set $\{0,...,n\} \times \{0,...,n\}$ with the metric $d((x_1,y_1),(x_2,y_2)) = |x_1 - x_2| + |y_1 - y_2|$. The diameter $D(L_n)$ is given by $d\big((0,0),(n,n)\big) = |n - 0| + |n - 0| = 2n$. The set of distances in $L_n$ is $\{0,...,2n\}$. We can find an explicit formula for $\Gamma_{L_n} : \{0,...,2n\} \to \mathbb{R}$. 
\begin{proposition} \label{prop:EGraph}
Let $L_n$ be a lattice graph. Then 
\[\Gamma_{L_n}(x) = \lfloor x/2 \rfloor\]
\end{proposition} 
\begin{proof}

Let $x \in \{0,...,2n\}$. We will first give an upper bound on $\Gamma_{L_n}(x)$. We start by showing that $\Gamma_{L_n}(x)$ is always an integer. Consider a quadruple $\{a,b,c,d\}$  in $L_n$ with coordinates $\{(x_a,y_a),(x_b,y_b),(x_c,y_c),(x_d,y_d)\}$. Note that negation and absolute value do not affect parity. So 
\[d(a,b) = |x_a - x_b| + |y_a - y_b| \equiv x_a + y_a + x_b + y_b \pmod{2}.\]
It follows that
\[d(a,b) + d(c,d) \equiv x_a + y_a + x_b + y_b + x_c + y_c + x_d + y_d\pmod{2}.\] 
We likewise have 
\[d(a,c) + d(b,d),d(a,d) + d(b,c) \equiv x_a + y_a + x_b + y_b + x_c + y_c + x_d + y_d\pmod{2}.\] 
Thus, the pairwise distance sums (see \cref{def:Def1}) all have the same parity. Then $L - M$ must be even, so $\delta = (L-M)/2$ is always an integer. This together with \cref{prop:maxg} implies that $\Gamma_{L_n}(x)$ is bounded above by the largest integer $N$ such that $N \leq x/2$. This upper bound explicitly is 
\[\Gamma_{L_n}(x) \leq N = \begin{cases}
\frac{x}{2} & \text{if $x$ is even} \\
\frac{x-1}{2} & \text{if $x$ is odd}
\end{cases} = \lfloor x/2 \rfloor \]

We now give a lower bound for $\Gamma_{L_n}(x)$.  \\
\textit{Case I:} Suppose $x = 2m, m \in \mathbb{Z}$. Consider the four points $(0,0), (0,m), (m, 0), (m, m)$ in $L_n$. The pairwise distance sums are $L = 4m, M = 2m, S = 2m$. Thus, 
\[\Gamma_{L_n}(x) \geq \frac{L-M}{2} = \frac{4m-2m}{2} = m  = \frac{x}{2}.\] 
\textit{Case II:} Suppose $x = 2m -1, m \in \mathbb{Z}$. Consider the four points $(0,0), (0,m), (m, 0), (m, m-1)$ in $L_n$. The pairwise distance sums are $L = 4m-1, M = 2m+1, S = 2m-1$. Thus, 
\[\Gamma_{L_n}(x) \geq \frac{L-M}{2} = \frac{4m-1-2m-1}{2} = \frac{2m-2}{2} = \frac{x-1}{2}.\]

\end{proof}

\subsection{Measuring Hyperbolicity in Graphs}\label{sec:def}

Recall our two model graphs: tree graphs, which model maximal hyperbolicity, and lattice graphs, which model minimal hyperbolicity (quasi-Euclidean geometry). The hyperbolicity of any graph $G$ should fall on a spectrum between these two extremes. Our goal is to characterize where along this spectrum $G$ falls. We do so as follows: \MDef
Recall from \cref{sec:mod}, that $\Gamma_T(x) = 0$ for a tree $T$. It follows that $s(T) = 0$. On the other hand, we have $\Gamma_{L_n}(x) = \lfloor x/2 \rfloor$ for a lattice graph $L_n$. It follows that 
\[s(L_n) = \frac{4\sum_{i=0}^{2n} \lfloor i/2 \rfloor}{2n(2n+1)} = \frac{4n^2}{2n(2n+1)} =   \frac{2n}{2n+1}.\]
This will converge to $1$ as $n$ increases. Moreover, \cref{prop:maxg} implies that $\Gamma_G(x) \leq x/2$ for any graph $G$. Thus, $s(G)$ will always lie in the interval $[0,1]$, between the two extremes of $s(T) = 0$ and $\lim_{n \to \infty} s(L_n) = 1$. We propose that the closer to $0$ that $s(G)$ lies in this interval, the greater the hyperbolicity $G$ is. 

We will motivate this by considering an analogue construction on the model space $\mathcal{M}^2_{-1/k^2}$, the unique complete, simply connected, Riemannian 2-manifold (surface) of constant sectional curvature $-1/k^2$. Note that $\mathcal{M}^2_{-1/k^2}$ is obtained from the rescaling the Riemannian metric on $\mathbb{H}^2$. Moreover, the distance metric induced by the Riemannian metric on $\mathcal{M}^2_{-1/k^2}$ is a constant multiple of the distance metric on $\mathbb{H}^2$: $d_{\mathcal{M}^2_{-1/k^2}}(a,b) = kd_{\mathbb{H}^2}(a,b)$ (see Chapters 2 and 6 of \cite{BridsonHaefliger1999}). It follows from the definition of $\Gamma$, $\delta(\square)$ and $D(\square)$ that $\Gamma_{\mathcal{M}^2_{-1/k^2}}(x) = k \Gamma_{\mathbb{H}^2}(x/k)$. 

Now consider $\Gamma_{\mathcal{M}^2_{-1/k^2}}|_{[0,1]}$. This represents the $\Gamma$ function associated to a bounded subset of $\mathcal{M}^2_{-1/k^2}$. This is done to mirror the bounded diameter of a graph. Note that $s(G)$ can be viewed as a discrete analogue to area under the curve $\Gamma_G$, rescaled so that $s(G)$ lies in $[0,1]$. For $k \in (0,\infty)$, let $F(k)$ be the area under the curve of $\mathbb{H}^2_k$ over $[0,1]$. This is analogous to $s(G)$ but without the rescaling. In \cref{prop:Fprop}, we will show that $F(k)$ is non-decreasing. So the smaller $F(k)$ is, the smaller $k$ is (which corresponds to increasingly negative curvature and therefore greater hyperbolicity). Since $F(k)$ was constructed by analogy with $s(G)$, we propose that a smaller value of $s(G)$ indicates a greater degree of hyperbolicity in $G$.

\begin{proposition}\label{prop:Fprop}
$F(k)$ is non-decreasing: we have 
\[F(k) \leq F(k')\]
for $k,k' \in (0,\infty)$ such that $k \leq k'$.
\end{proposition} 
\begin{proof}

Recall that $\Gamma_{\mathcal{M}^2_{-1/k^2}}(x) = k \Gamma_{\mathbb{H}^2}(x/k)$. 
It remains to show that 
\[F(k) = \int_{0}^1 \Gamma_{\mathcal{M}^2_{-1/k^2}}(t) dt = 
\int_{0}^1 k\Gamma_{\mathbb{H}^2}(t/k) dt \leq \int_{0}^1 k'\Gamma_{\mathbb{H}^2}(t/k') dt = F(k').\]
Since $\Gamma_{\mathbb{H}^2}$ is concave on $(0,\infty)$ by \cref{ther:Main} and continuous at $0$, it is concave on its domain. Since we also have $\Gamma_{\mathbb{H}^2}(0) = 0$, it follows that $\Gamma_{\mathbb{H}^2}(tx) \geq t\Gamma_{\mathbb{H}^2}(x)$ for $t \in [0,1]$. We have by construction that $k/k' \in [0,1]$. Then 
\[\Gamma_{\mathbb{H}^2}\left(\frac{x}{k'}\right) = \Gamma_{\mathbb{H}^2}\left(\frac{k}{k'}\frac{x}{k}\right) \geq \frac{k}{k'}\Gamma_{\mathbb{H}^2}\left(\frac{x}{k}\right).\] Which implies that $k \Gamma_{\mathbb{H}^2}(x/k) \leq k' \Gamma_{\mathbb{H}^2}(x/k')$. Integrating both sides of the inequality with respect to $x$ gives
$\int_{0}^1 k\Gamma_{\mathbb{H}^2}(t/k) dt \leq \int_{0}^1 k'\Gamma_{\mathbb{H}^2}(t/k') dt$, which completes the proof. 
\end{proof}

\section*{Acknowledgements}
This work is dedicated to the memory of Professor Ilya Zaliapin, whose mentorship inspired this research and whose support made it possible.

The author is especially grateful to Professor Tin-Yau Tam for his guidance and encouragement.

This research was supported by a Mission Support and Test Services, LLC (MSTS) Scholarship awarded to the author. He was also supported by NSF Grant No. 2122191 and the Ilya Zaliapin Student Conference Travel Award for FY 2026.

\printbibliography

\end{document}